\documentclass[12pt]{article}
\usepackage{hyperref} 

\usepackage{amsmath}
\usepackage{amssymb}
\usepackage{amscd}
\usepackage{textcomp}

\usepackage{amsthm}
\theoremstyle{plain}
\newtheorem{theorem}{Theorem}[section]

\newtheorem{lemma}[theorem]{Lemma}

\newtheorem{corollary}[theorem]{Corollary}
\newtheorem{proposition}[theorem]{Proposition}

       \newtheorem{definition-proposition}[theorem]{Definition-Proposition}
\theoremstyle{definition}
\newtheorem{definition}[theorem]{Definition}

\newtheorem{remark}[theorem]{Remark}

\theoremstyle{remark}

\usepackage{xy}
\xyoption{all}

\newdir^{ (}{{}*!/-3pt/\dir^{(}}    
\newdir^{  }{{}*!/-3pt/\dir^{}}    
\newdir_{ (}{{}*!/-3pt/\dir_{(}}    
\newdir_{  }{{}*!/-3pt/\dir_{}}

\newcounter{zahl}

\def\theenumi{(\alph{enumi})}

\def\p@enumii{\theenumi}

\newcommand{\DS}{\displaystyle}

\newcommand{\SC}{\scriptstyle}

\DeclareMathOperator{\Hom}{Hom}

\DeclareMathOperator{\charakt}{char}

\renewcommand{\phi}{\varphi}
\renewcommand{\epsilon}{\varepsilon}
\usepackage{amsfonts}
\newcommand{\BOne} {{\mathchoice{\hbox{\rm1\kern-2.7pt l\kern.9pt}}
                              {\hbox{\rm1\kern-2.7pt l\kern.9pt}}
                              {\hbox{\scriptsize\rm1\kern-2.3pt l\kern.4pt}}
                              {\hbox{\scriptsize\rm1\kern-2.4pt l\kern.5pt}}}}

\newcommand{\BA}{{\mathbb{A}}}

\newcommand{\BC}{{\mathbb{C}}}
\newcommand{\BD}{{\mathbb{D}}}

\newcommand{\BG}{{\mathbb{G}}}

\newcommand{\BQ}{{\mathbb{Q}}}

\newcommand{\BZ}{{\mathbb{Z}}}

\newcommand{\CG}{{\cal{G}}}

\newcommand{\CO}{{\cal{O}}}
\newcommand{\CP}{{\cal{P}}}

\let\setminus\smallsetminus

\newcommand{\ul}[1]{{\underline{#1}}}

\usepackage{ifthen}

\newcommand{\invlim}[1][]{\ifthenelse{\equal{#1}{}}% falls Argument leer
{\DS \lim_{\longleftarrow}}%                         verwende niedrige Version
{\DS \lim_{\underset{#1}{\longleftarrow}}}%  sonst:  verwende Argument
}

\newcommand{\dirlim}[1][]{\ifthenelse{\equal{#1}{}}% falls Argument leer
{\DS \lim_{\longrightarrow}}%                        verwende niedrige Version
{\DS \lim_{\underset{#1}{\longrightarrow}}}% sonst:  verwende Argument
}

\newcommand{\dotBD}{\vbox{\hbox{\kern2pt\bf.}\vskip-4.5pt\hbox{$\BD$}}}

\usepackage{color}
\newcounter{commentcounter}

\def\?{\ 
{\bf\color{red}???}\ 
\immediate\write16{}
\immediate\write16{Warning: There was still a question mark . . . }
\immediate\write16{}}

\def\isoto{\stackrel{}{\mbox{\hspace{1mm}\raisebox{+1.4mm}{$\SC\sim$}\hspace{-3.5mm}$\longrightarrow$}}}

\newbox\mybox
\def\arrover#1{\mathrel{
       \setbox\mybox=\hbox spread 1.4em{\hfil$\scriptstyle#1$\hfil}
       \vbox{\offinterlineskip\copy\mybox
             \hbox to\wd\mybox{\rightarrowfill}}}}

\begin{document}

%%%%%%%%%%%%%%%%%%%%%%%%%%%%%%%%%%%%%%%%%%%%%%%%%%%%%%%%%%%%%%%%%%%%%%

\author{Somayeh Habibi and Farhad Rahmati}

\title{A remark on a result of Huber and Kahn}

\maketitle

\begin{abstract}

A. Huber and B. Kahn construct a relative slice filtration on the motive $M(X)$ associated to a principal $T$-bundle $X\to Y$ for a smooth scheme $Y$. As a consequence of their result, one can observe that the mixed Tateness of the motive $M(Y)$ implies that the motive $M(X)$ is mixed Tate. In this note we prove the inverse implication for a principal $G$-bundle, for a split reductive group  $G$. \\
\noindent        
\it Mathematics Subject Classification (2000)\/: 14C25\\
\noindent
\textbf{Keywords:} Voevodsky motives; mixed Tate motives; slice filtration; $G$-bundles
\end{abstract}

\section*{Introduction}

The existence of a universal cohomology theory for schemes was first envisioned by A. Grothendieck in a letter to J.-P. Serre in 1964, see \cite{CGrothSerre}. He further developed the theory of pure Chow motives for smooth projective varieties as a means towards this purpose; e.g. see Demazure \cite{Dem} and Kleiman \cite{Klei}. Later, in the last decade of 20th century and first decade of 21st century, a remarkable progress in this theory has been achieved by V. Voevodsky et al., e.g. see \cite{VoeMor}, \cite{MVW}, \cite{Voe} and \cite{VoeBK}, who constructed several triangulated candidates for the category of motives over a perfect field $k$, such as
$$
DM_{gm}^{eff}(k), ~ DM_{-}^{eff}(k),~ DM_{-,et}^{eff}(k),~ DM_h(k)_\BQ ~\text{and etc.},
$$
that generalize the Grothendieck's construction beyond the smooth projective case. These are mainly categories of bounded complexes of sheaves with ``homotopy invariant” homology. As the main achievement of his approach, he proved the general Bloch–Kato conjecture \cite{VoeBK}.\\

In this paper we crucially use the machinery of slice filtration introduced by A. Huber and B. Kahn in \cite{H-K}. This theory gives a
well-behaved filtration on a motive $M$ in the triangulated category $DM_{-}^{eff}(k)$, whose successive cones are unique up to unique isomorphism and functorial in $M$, and furthermore, taking morphisms to $\BZ(n)$ for various $n$ gives a spectral sequence converging to its motivic cohomology. This filtration is particularly interesting when $M$ is mixed Tate (i.e. it lies in the subcategory generated by $\BZ(0)$ and the Tate object $\BZ(1)$), e.g. see \cite[Section 6]{H-K}. Note that the motive associated to many interesting varieties arising in the theory of algebraic groups, and also elsewhere, lie in the category of mixed Tate motives. The examples include, the motives associated to split reductive groups, see \cite{Big} and \cite{H-K}, projective homogeneous varieties, see\cite{Koeck}, Schubert varieties inside flag varities, and even Schubert varieties inside affine Grassmannian, see \cite{A-H} and also \cite[Subsection 4.4.1 and Section 5.4]{Hab}.\\

As an application of their theory,  A. Huber and B. Kahn construct a relative slice filtration on a principal $T$-bundle, for a split torus $T$, over a perfect field $k$, see \cite[Section 8]{H-K}. Consequently, one can see that for a principal $T$-bundle $X$ over a smooth variety $Y$, the motive $M(X)$ is mixed Tate if $M(Y)$ is mixed Tate. Using a geometric approach, in \cite{A-H}, the authors proved that (under some conditions) the motive of a $G$-bundle $X$ over a scheme $Y$, for a split reductive group $G$, lies in the category generated by mixed Tate motives and the motive corresponding to $Y$. Thus, in particular, the mixed Tateness of $M(Y)$ implies that $M(X)$ is mixed Tate. In this note we prove the inverse implication. Namely, we prove the following:

\begin{theorem}
Let $G$ be a split reductive group over a perfect field $k$. Let $X$ be a principal $G$-bundle over $Y$. Furthermore, assume either $X$ is locally trivial for the  Zariski topology on $Y$ or $k=\BC$. Then we have the following statements

\begin{enumerate}
\item[i)]
If $Y$ is stratified mixed Tate then $M(X)$ is mixed Tate.
\item[ii)]
If $Y$ is smooth and $M(X)$ is mixed Tate, then $M(Y)$ is mixed Tate. 
\end{enumerate}

In particular, when $k$ is a finite field and $M(Y)$ is stratified mixed Tate (resp. $M(X)$  is mixed Tate  and $Y$ is smooth), the $\BQ$-vector spaces $CH_i(X)$ (resp. $CH_i(Y)$) are finite.

\end{theorem}

\noindent
This is corollary \ref{Corollary} in the text. See definition \ref{DefSMT} for the definition of stratified mixed Tate.

This result has application in computing the motive associated to certain quotient varieties under the action of a reductive algebraic group. The authors will discuss this in a forthcoming work.

\section{Notation and Conventions}

Throughout this article we assume that $k$ is a perfect field. We denote by $\textbf{Sch}_k$ (resp. $\textbf{Sm}_k$) the category of schemes (resp. smooth schemes) of finite type over $k$.

%For $X$ in $\CO b(\textbf{Sch}_k)$, let $CH_i(X)$ and $CH^i(X)$ denote Fulton's $i$-th Chow groups and let $CH_\ast(X):=\oplus_i CH_i(X)$ (resp. $CH^\ast(X):=\oplus_i CH^i(X)$).  

\noindent
To denote the motivic categories over $k$, such as 

$$
\textbf{DM}_{gm}(k),~ \textbf{DM}_{gm}^{eff}(k),~ \textbf{DM}_{-}^{eff}(k),~ \textbf{DM}_{-}^{eff}(k) \otimes\mathbb{Q},~\text{etc.} 
$$  
and the functors $$M(-):\textbf{Sch}_k\rightarrow \textbf{DM}_{gm}^{eff}(k)$$ and $$M^c(-):\textbf{Sch}_k\rightarrow \textbf{DM}_{gm}^{eff}(k),$$ constructed by Voevodsky, we use the same notation that was introduced by him in \cite{Voe}. For the definition of the geometric motives with compact support in positive characteristics we also refer to \cite[Appendix B]{H-K}.\\

\noindent
Moreover for any object $M$ of $\textbf{DM}_{gm}(k)$ we denote by $M^*$ the internal Hom-object ${\underline{\text{Hom}}}_{\textbf{DM}_{gm}}(M,\BZ)$.\\

\noindent
For $X$ in $\CO b(\textbf{Sch}_k)$, we let $CH_i(X)$ and $CH^i(X)$ denote Fulton's $i$-th Chow groups.

\bigskip 

Note finally that throughout this article we either assume that $k$ admits resolution of singularities or we consider the motivic categories after passing to coefficients in $\mathbb{Q}$. \\

\section{Slice filtration and motive of $G$-bundles}

Let us first recall the definition of cellular varieties.

\begin{definition}
A variety is called \emph{cellular} if it contains
a cell (i.e. a variety isomorphic to an affine space) as an open subvariety such that the closed complement is cellular. 
\end{definition}

\begin{definition}\label{DefPureTate}
An object $M$ of $DM_{gm}(k)$ is called pure Tate motive if it is a (finite) direct sum of copies
of $\BZ(p)[2p]$ for $p\in \BZ$.
\end{definition}

\begin{remark}\label{RemHSareCellular}
\begin{enumerate}

\item
Using localization triangle, one can easily check that the motive associated to a smooth cellular variety is pure Tate. See also \cite[Cor. 3.6]{Kah}.
\item
Let $P$ be a parabolic subgroup of a split reductive group $G$. Let us recall that the action of Borel subgroup $B$ on the projective homogeneous variety $G\slash P$ induces a cell decomposition on it. E.g. see \cite[Thm 2.1]{Koeck}.
\end{enumerate}
\end{remark}

\begin{definition}
We denote by $TDM_{gm}^{eff}(k)$ the thick tensor subcategory of $DM_{gm}^{eff}(k)$, generated by $\mathbb{Z}(0)$ and the Tate object $\mathbb{Z}(1)$. An object of $TDM_{gm}^{eff}(k)$ is called a mixed Tate motive. $TDM_{-}^{eff}(k)$ is the localizing subcategory of $DM_{-}^{eff}(k)$ generated by $TDM_{gm}^{eff}(k)$.
\end{definition}

\begin{definition}\label{DefSMT}
A variety $X$ is called
\begin{enumerate}
\item
\emph{mixed Tate}, if $M^c(X) \in TDM_{gm}^{eff}(k)$.
\item
\emph{stratified mixed Tate}, if it admits a stratification $\{X_i\}_{i \in I}$, such that $X_i$ is smooth and mixed Tate, for every $i \in I$.
\end{enumerate}
\end{definition}

\begin{lemma}\label{Lem}
The motive $M^c(X)$ of a stratified mixed Tate variety $X$ is mixed Tate.
\end{lemma}
\begin{proof}
Let $\{X_i\}_{i\in I}$ be a stratification of $X$ such that each $X_i$ is smooth and mixed Tate. Consider the partial order on $I$ which is defined as follows:
$$
for~~i\neq j:~~~i <j~~~\text{if~and~only~if}~~~X_i\subseteq\overline{X_j}.
$$
\noindent
One defines the length of a stratification $\{X_i\}_{i \in I}$ to be
$$
\ell(\{X_i\}_{i \in I}):= max\{n;~ \exists~~ i_0 <i_1<\dots<i_n~~ \text{with}~ \{i_j\}_{j=0}^n\subseteq  I \}.
$$
\noindent

Now we prove the lemma by induction on the length of the stratification $\{X_i\}_{i\in I}$ of $X$. The base case is trivial. Assume that the length of the stratification is $n$, and the lemma holds for all stratifications with length $< n$. Set
$$
U:= \bigcup_j  X_j, 
$$
where $j$ runs over all maximal elements of $I$ with respect to the partial order. Since $X\setminus U$ is closed and $X = \overline{U}$, the subset $U$ is open and dense in $X$. Now the statement follows from the induction hypothesis and the following localization triangle
$$
M^c(X \setminus U) \rightarrow M^c(X) \rightarrow M^c(U) \rightarrow M^c(X \setminus U)[1].
$$

\end{proof}

Let us now recall the definition of the $n$-th slice filtration and motivic fundamental invariants from [HK].

\begin{definition}
Define the $n$-th slice filtration $\nu^{\geq n}M$ of a motive $M$ as the internal Hom object $\ul{\text{Hom}}(\BZ(n),M)(n)$.

\end{definition}

\noindent
The functor $c_n(-)$ is a triangulated endofunctor 
$$
c_n(-): DM_{-}^{eff}(k)\to DM_{-}^{eff}(k),
$$
which assigns the $n$-th motivic fundamental invariant $c_n(M)$, to a motive $M \in DM_{-}^{eff}(k)$, constructed by A. Huber and B. Kahn in \cite{H-K}. Below we briefly recall their construction.

\begin{definition}\label{DefSliceFilt}

\begin{enumerate}

\item
For $M$ in $\textbf{DM}_{-}^{eff}(k)$, by adjunction there is a morphism $a^n: \nu^{\geqslant n} M\to M$ (resp. $f^n:\nu^{\geqslant n} M\to \nu^{\geqslant n-1} M$), corresponding to identity
in
$$
\CD
\Hom(\ul\Hom(\BZ(n),M),\ul\Hom(\BZ(n),M))=\Hom(\nu^{\geqslant n} M, M),
\endCD
$$

\noindent
(resp. $=\Hom(\nu^{\geqslant n} M,\nu^{\geqslant n-1} M)$).

\item
Define $\nu_{< n}M = \nu_{\leq n-1}M$ using the following triangle

\begin{equation}\label{Huber_Kahn_Filt}
\xymatrix
   { 
      \nu^{\geqslant n} M \ar[rr]^{a^n} & & M\ar[dl]_{a_n}\\
     & \nu_{< n} M\ar[ul]_{~~~~~~~[1]}  &   
   }
\end{equation}

 The morphism $a_n : M\to \nu_{<n}M$ factors canonically through  $f_n : \nu_{<n+1}M \to \nu_{<n}M$, which allows to define

\begin{equation}\label{Huber_Kahn_Filt}
\xymatrix
   { 
      \nu_{<n+1}M \ar[rr]^{f_n} & & \nu_{<n}M\ar[dl]\\
     & \nu_{n}M\ar[ul]_{~~~~~~~[1]}  &   
   }
\end{equation}

\item
Finally define the following functors

 $$
c_n: \textbf{DM}_{gm}^{eff}(k) \to \textbf{DM}_{gm}^{eff}(k),
 $$

\noindent
where $\nu_n M=: c_n(M)(n)[2n]$. The object $c_n(M)$ is called the \emph{$n$-th motivic fundamental invariant} of $M$. 

\end{enumerate}
\end{definition}

In the following proposition we recall some basic properties of the motivic fundamental invariants.

\begin{proposition}\label{PropBasicPropertiesFMI}
The functor $c_n(-)$ has the following properties:
\newline 
\begin{enumerate}

\item
$c_n(M(1)[2])=c_{n-1}(M)$ for $M \in DM_{-}^{eff}(k)$, 
\item

For $M \in DM _-^{eff}(k)$ and $N \in TDM _-^{eff}(k)$, we have the K\"unneth isomorphism:
$$
\bigoplus_{p+q=n}c_p(M)\otimes c_q(N) \rightarrow c_n(M \otimes N).
$$

\item
Let $\charakt k=0$ and $X$ a $k$-variety of dimension $d$. Then $c_m M^c(X)=0$ for $m>d$ and $c_d M^c(X)=\ul{CH}_d(X)[0]$, 
where $\ul{CH}_d(X)$ denotes the homotopy invariant Nisnevich sheaf with transfers $U \mapsto CH_d(X \otimes _F F(U))$. Here $F(U)$ is the total ring of fractions of $U$. If $X$ is smooth, and with coefficients in $\BQ$, the assumption of characteristic $0$ is not
necessary.

\item
Let $X$ be a smooth variety such that $M(X)$ is a pure Tate motive. Then there is a natural isomorphism

$$
M(X)=\bigoplus_p c_p(M(X))(p)[2p],
$$

with $c_p(M(X)) = CH_p(X)^\ast[0]$. Here $_-^{~\ast}$ denotes the dual of a free abelian group.

\end{enumerate}
\end{proposition}

\begin{proof}
Part a) follows from definition. For part b) see  \cite[lem~4.8]{H-K} and for c) see \cite[Prop~1.7 and Prop~1.8]{H-K}. For the last part see \cite[Prop 4.10]{H-K}.
\end{proof}

Let $D^b(Ab)$ be the bounded derived category of abelian groups $Ab$. Denote by $D_f^b(Ab)$ its full subcategory consisting of those objects whose cohomology groups are finitely generated. Consider the fully faithful triangulated functor 
$$
\iota: D_f^b(Ab) \to DM _{gm}^{eff}(k),
$$ 
which sends $\mathbb{Z}$ to $\mathbb{Z}[0]$ and respect tensor structures. Its essential image is the thick tensor subcategory of $DM _{gm}^{eff}(k)$ generated by $\mathbb{Z}(0)$. Similarly there is a fully faithful tensor functor $\iota: D^-(Ab) \to DM_{-}^{eff}(k)$ from bounded above derived category $D^-(Ab)$ to $DM_{-}^{eff}(k)$, whose essential image is generated by $\mathbb{Z}(0)$.

\bigskip

The following proposition \cite[Prop 4.6]{H-K}, gives a criterion for mixed Tateness of a motive $M$, in terms of the associated fundamental motivic invariants $c_n(M)$.

\begin{proposition}\label{PropCriterionForMT}
Assume either $chark=0$ or coefficients in $\mathbb{Q}$. A motive $M \in DM_{gm}^{eff}(k)$ is in $TDM_{gm}^{eff}(k)$ if and only if $c_n(M) \in D_f^b(Ab)$ for all $n$ and $c_n(M)=0$ for $n$ large enough. If $M \in  TDM_{-}^{eff}(k)$, then $c_n(M) \in D^-(Ab)$.
\end{proposition}

In \cite{H-K} theorem 8.8 the authors construct a relative slice filtration on a torus bundle. They use this filtration to compute the motive associated to a split reductive group $G$. Below we recall their result. 

\begin{proposition}\label{PropRSFTB}
Let $T$ be a split torus of dimension $r$ and $X$ a principal $T$ bundle over a smooth variety $Y$ over a field $k$. There is a filtration
$$
\nu_Y^{\geqslant p+1}M(X)\rightarrow \nu_Y^{\geqslant p}M(X) \rightarrow ... \rightarrow M(X)
$$
in $DM_{gm}^{eff}(k)$, where $M(X) \cong \nu^{\geqslant 0}M(X)$, $\nu^{\geqslant r+1}M(X)=0$, together with distinguished triangles

\begin{equation}\label{Huber_Kahn_Filt}
\xymatrix
   { 
      \nu_Y^{\geqslant p+1}M(X) \ar[rr] & & \nu_Y^{\geqslant p}M(X)\ar[dl]\\
     & M(Y)(p)[p]\otimes\wedge^p(\Xi)\ar[ul]_{~~~~~~~~~~~~~~~~~~[1]}  &   
   }
\end{equation}
for $0 \leq p \leq r$. Here $\Xi= Hom (\mathbb{G}_m, T)$ is the cocharacter group.
\end{proposition}

\begin{remark}
Note moreover that in \cite[$\S$4]{A-H} the authors construct a nested filtration on the motive associated to a $G$-bundle using the theory of wonderful compactification of reductive groups.
\end{remark}

\begin{theorem}\label{MainThm}
Let $T$ be a torus of dimension $r$. For a $T$-bundle $X$ over $Y$, we have the following statements:
\begin{enumerate}
\item[i)] 
If $Y$ is stratified mixed Tate, then $X$ is mixed Tate.
\item[ii)]
Assume that $Y$ is smooth. If $X$ is mixed Tate, then $Y$ is mixed Tate. 
\end{enumerate}
\end{theorem}

\begin{proof}

First we prove ii). For this let us first assume that $r=1$, i.e. $T=\BG_m$, and $X$ is a $\mathbb{G}_m$-bundle over $Y$. For the $\mathbb{P}^1$-bundle $X \times ^T \mathbb{P}^1= \CP$, consider the following exact triangle
$$
M^c(X) \rightarrow  M^c(\CP) \rightarrow  M^c(\CP\setminus X) \rightarrow  M^c(X)[1].
$$
By projective bundle formula \cite[Thm 15.12]{MVW}, we get
$$
M^c(X) \rightarrow M^c(Y)\oplus M^c(Y)(1)[2] \rightarrow (M^c(Y))^{\oplus 2} \rightarrow  M^c(X)[1].
$$ 
Applying functor $c_n(-)$ to the above triangle we get 
$$
c_n(M^c(X)) \rightarrow c_n(M^c(Y))\oplus c_{n-1}(M^c(Y)) \rightarrow (c_n(M^c(Y)))^{\oplus 2} \rightarrow  c_n(M^c(X)[1]).
$$ 
Note that $c_n(M(1)[2])=c_{n-1}(M)$, see proposition \ref{PropBasicPropertiesFMI} a). For $n>\dim Y$, we have $c_n(M^c(Y))=0$, see proposition \ref{PropBasicPropertiesFMI} c), and therefore we have
$$
c_n(M^c(X))\isoto c_{n-1}(M^c(Y)).
$$
Since $X$ is mixed Tate, by proposition \ref{PropCriterionForMT} we see that $c_n((M^c(X))$, and hence $c_{n-1}(M^c(Y))$, lie in $D_f^b(Ab)$. Now, regarding the following triangle
$$
c_{n-1}(M^c(X)) \rightarrow c_{n-1}(M^c(Y))\oplus c_{n-2}(M^c(Y)) \rightarrow (c_{n-1}(M^c(Y)))^{\oplus 2},
$$ 
and proposition \ref{PropCriterionForMT}, we observe that $c_{n-2}(M^c(Y))$ also lies in $D_f^b(Ab)$. Following this way, we can argue recursively that $c_i(M^c(Y))$ lie in $D_f^b(Ab)$ for all $i<n$. 
\\
Note that the general case reduces to the case where $r=1$. Namely, recall that for a principal $G$-bundle $\CG \rightarrow Y$, and a closed subgroup scheme $H$ of the reductive algebraic group $G$, the quotient map $\CG \rightarrow \CG/H\cong \CG\times^G(G/H)$ is a principal $H$-bundle. Regarding this, we can view the $T$-bundle $X$ as a sequence $$X=:X_r \rightarrow X_{r-1}  \rightarrow \dots \rightarrow X_2 \rightarrow X_1\rightarrow X_0:=Y,$$    
such that $X_i$ is a $\mathbb{G}_m$-bundle over $X_{i-1}$ for $1 \leq  i \leq r$.\\

\noindent
To prove i) let us first assume that $Y$ is smooth. In this case we have the following sequence

\begin{equation}\label{Huber_Kahn_Filt}
\xymatrix @C=-0.2pc {
\nu^{\geqslant p+1}_Y M(X) \ar[rrrrrr]& & & & & & \nu^{\geqslant p}_Y M(X) \ar[dd]& && \nu ^{\geqslant2}_Y M(X) \ar[rrrrr]& & & && \nu^{\geqslant 1}_Y M(X) \ar[rrrrrr]\ar[dd]& & && & & M(X)\ar[dd]  \\
& & & & & & & & & ... & &  & &  & && & & & & \\
 & & & & & & \lambda_p(Y,T)\ar[uullllll]_{~~~[1]}& & & &  & & & & \lambda_1(Y,T)\ar[uulllll]_{~~[1]} & & & & & & \lambda_0(Y,T)\ar[uullllll]_{~~[1]}
}
\end{equation}
of exact triangles, according to proposition \ref{PropRSFTB}. Here 
$$
\lambda_p(Y,T):= M(Y)(p)[p] \otimes \Lambda^p(\Xi),
$$ 
for $0 \leq  p \leq  r$. Recall that $M(X) \cong \nu_Y^{\geqslant 0}M(X)$ and $\nu_Y^{\geqslant r+1}M(X)=0$.\\
Since $\nu_Y^{\geqslant r+1}M(X)=0$ we get an isomorphism 
$$\nu_Y^{\geqslant r}M(X)\isoto M(Y)(p)[p]\otimes\Lambda^p(\Xi).
$$ 
Therefore $\nu_Y^{\geqslant r}M(X)$ is mixed Tate. This is because $Y$ is smooth and mixed Tate, and that $M(Y)^\ast\cong M^c(Y)(-n)[-2n]$. Note that for mixed Tate motives $M$ and $N$ in $DM_{gm}^{eff}(k)$, the Hom object $\ul{Hom}(M,N)$ is mixed Tate. This follows from the fact that $\ul{Hom}(\BZ(i),\BZ(j))$ equals $\BZ(j-i)$ if $i\leq j$ and vanishes elsewise. In particular we see that $M$ is mixed Tate if the dual object $M^\ast$ is mixed Tate, which justifies that for smooth schemes the definition \ref{DefSMT}. agrees with the alternative definition using the functor $M(-)$.\\ 
Regarding the above discussion and the diagram \ref{Huber_Kahn_Filt}, we may argue recursively  that $\nu_Y^{\geqslant i}M(X)$ are mixed Tate, for every $i\geq 0$.\\ 

Now, assume that $Y$ is stratified mixed Tate, with stratification $\{Y_i\}_i$. Then, by the above arguments we see that $X_i:=X\times_Y Y_i$ are mixed Tate and hence $X$ is stratified mixed Tate. Thus we may conclude by lemma \ref{Lem} 
\end{proof}

\begin{corollary}\label{Corollary}
Let $G$ be a split reductive group over a perfect field $k$. Let $X$ be a principal $G$-bundle over $Y$. Furthermore, assume either $X$ is locally trivial for the  Zariski topology on $Y$ or $k=\BC$. Then we have the following statements

\begin{enumerate}
\item[i)]
If $Y$ is stratified mixed Tate then $M(X)$ is mixed Tate.
\item[ii)]
If $Y$ is smooth and $M(X)$ is mixed Tate then $M(Y)$ is mixed Tate. 
\end{enumerate}
In particular, when $k$ is a finite field and $M(Y)$ is stratified mixed Tate (resp. $M(X)$  is mixed Tate  and $Y$ is smooth), the $\BQ$-vector spaces $CH_i(X)$ (resp. $CH_i(Y)$) are finite.

\end{corollary}

\begin{proof}

Assume that $Y$ is smooth. Let $T$ be the maximal torus in $G$ and let $B$ be a Borel subgroup containing $T$. Applying the above theorem to the $T$-bundle $X\to X\slash T$, we see that $X$ is mixed Tate if and only if $X\slash T$ is mixed Tate. Note that $B=TU$ for the unipotent group $U$, and thus $X\slash T\to X\slash B$ is an affne bundle. The latter is a projective homogeneous fiberation over $Y$. Now the statement follows from remark \ref{RemHSareCellular} b), proposition \ref{PropBasicPropertiesFMI} d), motivic Leray-Hirsch theorem \cite[Thm 2.8]{A-H}, when $X$ is locally trivial for the Zariski topology, and from \cite[Thm 1.1.b)]{Iyer} when $k=\BC$. Note that for part $i)$, regarding lemma \ref{Lem}, we may reduce to the case that $Y$ is smooth and mixed Tate. \\
For the remaining statement see \cite[Prop 3.9]{A-H2}.
\end{proof}

\paragraph{Acknowledgment:} The first author warmly thanks B. Kahn for very useful comments on the earlier draft of the paper. She thanks J. Ayoub and E. Arasteh Rad for useful comments and discussions. She is also grateful to Prof. Luca Barbieri Viale for steady encouragement.\\ This research was in part supported by a grant from IPM (No.1402140032). 
%%%%%%%%%%%%%%%%%%%%%%%%%%%%%%%%%%%%%%%%%%%%%%%%%%%%%%%%%%%%%%%%%%%%%%
%
%    Bibliography
%
%%%%%%%%%%%%%%%%%%%%%%%%%%%%%%%%%%%%%%%%%%%%%%%%%%%%%%%%%%%%%%%%%%%%%%

%%%%%%%%%%%%%%%%%%%%%%%%%%%%%%%%%%%%%%%%%%%%%%%%%%%%%%%%%%%%%%%%%%%%%%
%
%    Bibliography
%
%%%%%%%%%%%%%%%%%%%%%%%%%%%%%%%%%%%%%%%%%%%%%%%%%%%%%%%%%%%%%%%%%%%%%%

\begin{minipage}[t]{1\linewidth}
\noindent
\small\textbf{Somayeh Habibi},\\ 
-Faculty of Mathematics, Tehran Polytechnic Univ., 424 Hafez Ave., Tehran 15914\\
-School of Mathematics, Institute for Research in Fundamental Sciences (IPM), P.O. Box: 19395-5746, Tehran, Iran
 email: \href{shabibi@ipm.ir}{shabibi@ipm.ir}

\bigskip

\noindent
\small\textbf{Farhad Rahmati}, Faculty of Mathematics, Tehran Polytechnic Univ., 424 Hafez Ave., Tehran 15914, Iran, email: \href{frahmati@aut.ac.ir}{frahmati@aut.ac.ir}

\end{minipage}


{\small
\begin{thebibliography}{GHKR2}
\addcontentsline{toc}{section}{References}



\bibitem{CGrothSerre} Correspondance Grothendieck-Serre. Edited by Pierre Colmez and Jean-Pierre Serre. Documents
Math\'ematiques (Paris), 2. Soc. Math. France, Paris, 2001,


\bibitem{A-H} S. Habibi, E. Arasteh Rad, \emph{On the motive of fibre bundle and its applications}, Analysis, geometry and number theory, 2017, Pisa : Fabrizio Serra, 2017 , 77-96,

\bibitem{A-H2} S. Habibi, E. Arasteh Rad, \emph{Motivic remarks on the moduli stacks of global G-shtukas and their local models}, preprint available at \href{1912.09968}{https://arxiv.org/abs/1912.09968},

\bibitem{Big} S. Biglari, \emph{Motives of reductive groups}, American Journal of Mathematics, 2012, 134(1), 235--257.

\bibitem{Dem} M. Demazure, \emph{Motifs des vari\'et\'es alg\'ebriques}, S\'eminaire Bourbaki 1969/70, Expos\'e 365, 20pp.
\bibitem{Klei} S. L. Kleiman, \emph{Motives. Algebraic geometry}, Oslo 1970, Proc. Fifth Nordic Summer-School in Math., Oslo, 1970, pp. 53–82. Wolters-Noordhoff, Groningen, 1972.


\bibitem{Hab} S. Habibi, \emph{On the Motive of a Bundle},  PhD Thesis, Universita degli Studi di Milano, 2012, available as \url{https://air.unimi.it/retrieve/handle/2434/212311/197090/phd_unimi_R08381.pdf},


\bibitem{H-K} A. Huber, B. Kahn, \emph{The slice filtration and mixed Tate motives}. Compositio Mathematica, 142(4), 907-936. doi:10.1112/S0010437X06002107,


\bibitem{Iyer} J. Iyer, \emph{Absolute Chow-K\"unneth decomposition for rational homogeneous 
bundles and for log homogeneous varieties}. Michigan Math. J., 60(1):79--91 (2011),


\bibitem{Kah} B. Kahn, \emph{Motivic cohomology of smooth geometrically cellular varieties}, in Algebraic K-theory, Seattle, 1997, Proceedings of Symposia in Pure Mathematics, vol. 67 American Mathematical Society, Providence, RI, 1999), 149--174,


\bibitem{Koeck} B. K\"ock, \emph{Chow motif and higher Chow theory of $G\slash P$}, Manuscripta Math. 70 (1991), 363--372,


\bibitem{VoeMor} F. Morel, V. Voevodsky, \emph{$\BA^1$-homotopy theory of schemes}. Inst. Hautes Etudes Sci. Publ. Math., (90):45--143
(2001), 1999.

\bibitem{MVW} 
C.\ Mazza, V.\ Voevodsky, C.\ A.\ Weibel. \emph{Lecture notes on motivic cohomology}, Clay mathematics monographs, v.2. (2006),




\bibitem{Voe} V. Voevodsky, \emph{Triangulated categories of motives over a field}, in Cycles, transfers and
motivic cohomology theories, Annals of Mathematics Studies, vol. 143 (Princeton University Press,
Princeton, NJ, 2000), 188–238,

\bibitem{VoeBK} V. Voevodsky, \emph{On motivic cohomology with $\BZ/\ell$-coefficients}, Ann. of
Math. (2) 174 (2011), no. 1, 401--438.





\end{thebibliography}



}
\end{document}